\documentclass{amsart}
\usepackage{amsmath,amssymb,amsfonts,enumerate,amsthm,graphicx,tikz,graphics}

\newtheorem{thm}{Theorem}[section]

\newtheorem{lem}[thm]{Lemma}
\newtheorem{prop}[thm]{Proposition}
\newtheorem{defn}[thm]{Definition}

\newtheorem{exam}[thm]{Example}
\newtheorem{rem}[thm]{Remark}

\numberwithin{equation}{section}

\begin{document}
\bibliographystyle{amsplain}

\title{A canonical polytopal resolution for transversal monomial ideals}

\author[Rahim Zaare-Nahandi]{Rahim Zaare-Nahandi\\University of Tehran}
\address{}

\thanks{}



\begin{abstract}
\noindent Let $S = k[x_{11}, \cdots, x_{1b_1}, \cdots, x_{n1},
\cdots, x_{nb_n}]$ be a polynomial ring in $m = b_1 + \cdots + b_n$
variables over a field $k$. For all $j$, $1\le j \le n$, let $P_j$
be the prime ideal generated by variables $\{x_{j1}, \cdots,
x_{jb_j}\}$ and let
$$I_{n, t} = \sum_{1\le j_1< \cdots <j_t\le n} P_{j_1}\ldots P_{j_t}$$
be the {\it transversal monomial ideal of degree $t$ on} $P_1,
\cdots, P_n$. We explicitly construct a canonical polytopal
$\mathbb{Z}^t$-graded minimal free resolution for the ideal $I_{n,
t}$ by means of suitable gluing of polytopes.
\end{abstract}

\maketitle

\section{Introduction}
The idea to describe a resolution of a monomial ideal by means of
combinatorial-geometric chain complexes was initiated by Bayer,
Peeva and Sturmfels \cite{BPS98}, and was extended by Bayer and
Sturmfels \cite{BS98}. Novik, Postnikov and Sturmfels gave a
polytopal complex that supports the minimal free resolution of
matroidal ideals \cite{NPS02}. Further extension was made by
J\"ollembeck and Welker who used discrete Morse theory to construct
desired CW complexes \cite{JW09}. Sinefakoupols showed that the
Eliahou-Kervaire resolution \cite{EK90} of Borel-fixed monomial
ideals which are Borel-fixed generated in one degree are cellular
supported on a union of convex polytopes \cite[Theorem 20]{Sin08}.
Mermin and Clark, separately, constructed regular cell complexes
which support the Eliahou-Kervaire resolution of a stable monomial
ideal \cite{Mer10}, \cite{C12}.
On the other hand, Valesco showed that there exist minimal free
monomial resolutions which can not be supported on any CW-complex
\cite{V08}. Still from a different point of view, Fl{\o}ystad
investigated monomial labelings on cell complexes which give minimal
free resolution of the ideal generated by these monomials in the
case the quotient ring is Cohen-Macaulay \cite{F09}. Dochtermann and
Engstr\"om gave a cellular resolutions for the ideals of cointerval
hypergraphs supported on polyhedral complexes and, they extended
their construction to more general hypergraphs \cite[Theorems 4.4
and 6.1]{DE10}. Further on, Engstr\"om and Nor\'en constructed
cellular resolutions for powers of ideals of a bipartite graph
\cite[Proposition 4.4 and Theorem 7.2]{DE12}. Okazaki and Yanagawa
showed that the Eliahou-Kervaire resolution of a Cohen-Macaulay
stable monomial ideal is supported by a regular CW complex whose
underlying space is a closed ball and they improved their result to
some variants of Bore-fixed ideals \cite[Theorem 3.8]{OY13}. More
recently, Goodarzi proved that the Herzog-Takayama resolution
\cite{HT02} is cellular \cite{G15}. This is a resolution for the
class of monomial ideals with regular linear
quotients. This class includes matroidal ideals and stable monomial ideals.\\


In this paper we consider the class of transversal monomial ideals
of degree $t$ (see Definition \ref{Defn}). A minimal free resolution
for monomial ideals in this class was given in \cite{ZZ04} and
\cite{Zaa11}. However, in general, this resolution is not polytopal
(see Remark \ref{compare}). We construct a labeled polytopal cell
complex which provides a canonical $\mathbb{Z}^t$-graded minimal
free resolution for transversal monomial ideals of degree $t$. Any
such monomial ideal is in fact a matroidal ideal (see Remark \ref
{Matr}). Thus, by \cite{G15} its Herzog-Takayama resolution is
cellular. Nevertheless, the Herzog-Takayama resolution is obtained
by consecutive applications of mapping cone constructions. In
contrast, the resolution provided here is based on decomposing any
ideal in this class and ``gluing'' together the polytopal cell
complexes associated to each summand to get a polytopal cell complex
which gives the expected resolution (see Lemma \ref{TLemma} and
Theorem \ref{TThm}). This resolution is explicit and is much simpler
to compute compared to the Herzog-Takayama resolution (see Example
\ref{TExam}). Furthermore, unlike most results which prove
cellularity of certain known resolutions, the cell-complex
constructed here provides the resolution.\\

The paper is organized is as follows. After the introduction in
Section 1, in Section 2 we provide preliminary material and
notations. Section 3 is devoted to the main result, the construction
of a polytopal cell complex that supports a graded minimal free
resolution of a transversal monomial ideal. Then the square-free
Veronese ideals of fixed degree, the only Cohen-Macaulay case of
this class of ideals, is treated. In this case, the cell complex is
a polytopal subdivision of a simplex, in particular, its underlying
space is homeomorphic to a closed ball. This last result, using some
``depolarization'', also follows by a result of Sinefakopoulos
\cite{Sin08}.

\section{Notations and preliminaries}
We assume familiarity with basic notions of cell complexes,
polytopes and monomial ideals referring to \cite{Zie95},
\cite{BH98}, \cite{MS05} and \cite{HH11}. In particular, we have
freely used several statements and results from \cite{Sin08}. Let
$I\subset R = \mathbf{k}[y_1,\cdots, y_n]$ be a monomial ideal in
the polynomial ring over a field $\mathbf{k}$ and let $G(I)$ be the
unique minimal monomial generating set of $I$. Let $X$ be a regular
cell complex with vertices labeled by members of $G(I)$. Let
$\epsilon_X$ be an incidence function on $X$. Any face of $X$ will
be labeled by $\mathbf{m}_F$, the least common multiple of the
monomials in $G(I)$ on the vertices of $F$. If $\mathbf{m}_F =
y_1^{a_1}\cdots y_n^{a_n}$, then the {\it degree} $\mathbf{a}_F$ is
defined to be the exponent vector $e(\mathbf{m}_F) =
(a_1,\cdots,a_n)$. Let $RF$ be the free $R$-module with one
generator in degree $\mathbf{a}_F$. The {\it cellular complex}
$\mathbf{F}_X$ is the $\mathbb{Z}^n$-graded $R$-module
$\bigoplus_{\emptyset\ne F\in X} RF$ with differentials
$$\partial(F) = \sum_{\emptyset\ne F'\in X} \epsilon(F,F')
\frac{\mathbf{m}_F}{\mathbf{m}_{F'}} F',$$ where $\epsilon(F,F')$ is
nonzero only if $F'$ is a face of $F$ of codimension $1$ and is $+1$
or $-1$ such that the resulting sequence is a chain complex
(see \cite{BS98}).\\

If the complex $\mathbf{F}_X$ is is exact, then $\mathbf{F}_X$ is
called a cellular resolution of $I$. Alternatively, we say that $I$
has a cellular resolution supported on the labeled cell complex $X$.
If $X$ is a simplicial complex or a polytopal cell complex, then
$\mathbf{F}_X$ is called {\it simplicial} and {\it polytopal}
resolution, respectively. A cellular resolution $\mathbf{F}_X$ is
minimal if and only if any two comparable faces $F'\subset F$ of the
same degree coincide. For more on cellular resolutions and polytopal
complexes we refer to \cite{BH98}, \cite{MS05} and
\cite{Zie95}.\\

The following two lemmas will be essential for our constructions.

\begin{lem}\label{Union} (The gluing lemma) \cite[Lemma 6]{Sin08}. Let $I$ and
$J$ be two monomial ideals in $R$ such that $G(I + J) = G(I) \cup
G(J)$.
Suppose that\\
$(i)$ $X$ and $Y$ are labeled regular cell complexes in some
$\mathbb{R}^N$ that supports a minimal free resolution
$\mathbf{F}_X$ and $\mathbf{F}_Y$ of $I$ and $J$, respectively,
and\\
$(ii)$ $X\cap Y$ is a labeled regular cell complex that supports a
minimal free resolution $\mathbf{F}_{X\cap Y}$ of $I\cap J$.\\
Then $X\cup Y$ is a labeled regular cell complex that supports a
minimal free resolution of $I + J$.
\end{lem}

\begin{rem}\cite[Remark 7]{Sin08}.
For any two monomial ideals $I$ and $J$, we have
$$G(I + J)\subseteq G(I) \cup G(J).$$
A case where equality holds is when all elements of $G(I) \cup G(J)$
are of the same degree, or more generally, when $G(I) \cup G(J)$ is
a minimal set of generators.
\end{rem}

\begin{lem}\label{Product} \cite[Lemma 8]{Sin08}.
Let $I \subset \mathbf{k}[y_1,\cdots, y_k]$ and $J \subset
\mathbf{k}[y_{k+1},\cdots, y_n]$ be two monomial ideals. Suppose
that $X$ and $Y$ are labeled regular cell complexes in some
$\mathbb{R}^N$ of dimensions $k-1$ and $n-k-1$, respectively, that
support minimal free resolutions $\mathbf{F}_X$ and $\mathbf{F}_Y$
of $I$ and $J$, respectively. Then the labeled cell complex $X\times
Y$ supports a minimal free resolution $\mathbf{F}_{X\times Y}$ of
$IJ$.
\end{lem}

Recall that a simplicial complex $\mathcal{M}$ on $[n] = \{1,\cdots,
n\}$ is called a {\it matroid} if, for any two distinct facets $F$
and $G$ of $\mathcal{M}$ and for any $i\in F$, there exists $j\in G$
such that $(F\setminus\{i\})\cup \{j\}$ is a facet of $\mathcal{M}$.
A matroid $\mathcal{M}$ is said to have {\it strong exchange
property} if, for any two distinct facets $F$ and $G$ of
$\mathcal{M}$ and for any $i\in F$, $j\in G \setminus F$,
$(F\setminus\{i\})\cup \{j\}$ is a facet of $\mathcal{M}$. Let $A =
\{v_1,\cdots, v_m\}$ be a set of vectors in $\mathbb{R}^n$. The set
of all linearly independent subsets of $A$ of cardinality $t \le n$
forms a matroid which is called a {\it linear matroid}. A matroid is
called linear if it has a representation using a vector space over a
field. A monomial ideal $I\subset \mathbf{k}[y_1,\cdots, y_n]$ is
said to be a {\it matroidal ideal} if it is the facet ideal of a
matroid $\mathcal{M}$ on $[n]$ under some labeling of vertices of
$\mathcal{M}$ by the variables $y_1, \cdots, y_n$.\\

\begin{defn}\label{Defn} We fix positive integers $1 \le b_1\le \cdots \le b_n$ and a
polynomial ring $$S = k[x_{11}, \cdots, x_{1b_1}, \cdots, x_{n1},
\cdots, x_{nb_n}]$$ in $m = b_1 + \cdots + b_n$ variables over a
field $k$. Let $P_j = (x_{j1}, \cdots, x_{jb_j})$, $j = 1, \cdots,
n$. The ideal
$$I_{n, t} = \sum_{1\le j_1< \cdots <j_t\le n}
P_{j_1}\ldots P_{j_t} = \sum_{1\le j_1<\cdots <j_t\le n} P_{j_1}\cap
\ldots \cap P_{j_t}$$ is called {\it the transversal monomial ideal
of degree $t$ on} $P_1, \cdots, P_n$.
\end{defn}

This class of ideals naturally arises in the study of generic
singularities of algebraic varieties
\cite{SZ05}.\\

It is helpful to notice that the ideal $I_{n, t}$ may also be
realized as the ideal of $t$-minors of a matrix $D$ where
$$D=\left[
\begin{array}{cccccccccccc}
x_{11}& \cdots& x_{1b_1}&0&\cdots&0&0&\cdots&0&0&\cdots&0\\
0&\cdots&0&x_{21} &\cdots&x_{2b_2}&0&\cdots&0&0&\cdots&0\\
\vdots&\vdots&\vdots&
\vdots&\vdots&\vdots&\vdots&\vdots&\vdots&\vdots&\vdots&\vdots\\
0&\cdots&0&0&\cdots&0&0&\cdots&0&x_{n1}&\cdots&x_{nb_n}
\end{array}\right].$$

\begin{rem}\label{Matr} Let $\{v_1, \cdots, v_n\}$ be a vector space basis in
$\mathbb{R}^n$. For each $i$, $1 \le i \le n$, let $\lambda_{ij}$,
$j = 1, \cdots, b_i$, be distinct nonzero real numbers. Let $A =
\{\lambda_{ij} v_i: 1 \le i \le n, 1 \le j \le b_i\}$. In other
words, $A$ is a set of vectors linearly independent up to
proportionality. Let $\mathcal{M}$ be the linear matroid of all
linearly independent subsets of $A$ of cardinality $t \le n$. Then
$I_{n, t}$ is the facet ideal of the linear matroid $\mathcal{M}$,
when the vertex corresponding to $\lambda_{ij}v_i$ is labeled by
$x_{ij}$. In particular, $I_{n, t}$ is a matroidal ideal. Therefore,
$I_{n,t}$ has linear quotients \cite[Theorem 12.6.2]{HH11}, and a
linear resolution \cite{Bj92} or \cite[Lemma 8.2.1]{HH11}. On the
other hand, the only case $I_{n,t}$ is Cohen-Macaualay is the case
of the square-free Veronese ideal of degree $t$, i.e., the case $b_1
= \cdots = b_n = 1$, $n=m$ \cite[Theorem 12.6.7]{HH11}. In this case
the corresponding matroid $\mathcal{M}$ has the strong exchange
property.
\end{rem}

\section{A cellular graded resolution for transversal monomial ideals}

In this section we construct a polytopal
cell complex that supports a graded minimal free resolution of $I_{n, t}$.\\

We will need the construction for certain special cases.

\begin{prop}\label{TProp} Let $\Delta(P_i)$ be the
$(b_i-1)$-simplex with vertices labeled by variables generating
$P_i$. Then the following statements
hold:\\
(a) The labeled cartesian product of simplexes $\Gamma_{n,n} =
\Delta(P_1)\times \cdots \times \Delta(P_t)$ supports a
$\mathbb{Z}^n$-graded minimal free resolution of $I_{n,n} =
P_1.\cdots .P_n.$\\
(b) $I_{n,n-1} = \sum_{i=1}^{n-1}
P_1.\cdots.P_{i-1}.(P_i+P_{i+1}).P_{i+2}.\cdots.P_n$ , and the
labeled cell complex
$$\Gamma_{n,n-1} = \bigcup_{i=1}^{n-1}\Delta(P_1)\times\cdots\times\Delta(P_{i-1})\times
\Delta(P_i+P_{i+1})\times\Delta(P_{i+2})\times\cdots\times\Delta(P_n)$$
supports a $\mathbb{Z}^{n-1}$-graded minimal free resolution of
$I_{n,n-1}$ and $\Gamma_{n,n-1}\subset \Gamma_{n,n}$.
\end{prop}
\begin{proof} (a) The Koszul complex on the sequence $\{x_{i1}, \cdots,
x_{ib_i}\}$ is a minimal free resolution of $P_i$, $i = 1, \cdots,
n$ and $\Delta(P_i)$ supports the minimal free resolution of $P_i$.
Hence by Lemma 2.3 the cartesian product $\Delta(P_1)\times \cdots
\Delta(P_n)$, as a labeled polytopal cell complex, supports a
$\mathbb{Z}^n$-graded minimal free resolution of $I_{n,n}$.\\

\noindent (b) We use induction on $n$ for both claims. A generator
of $I_{n,n-1}$ is either a monomial with a variable in $P_n$ or with
a variable in $P_{n-1}+P_n$. Thus we have
$$I_{n,n-1} = I_{n-1,n-2}.P_n + I_{n-2,n-2}.(P_{n-1}+P_n).$$
Replacing $I_{n-1,n-2}$ by induction hypothesis and $I_{n-2,n-2}$ by
(a) the equality for $I_{n,n-1}$ follows.\\
Now by induction hypothesis and Lemma 2.3, $\Gamma_{n-1,n-2}\times
\Delta(P_n)$ and $\Gamma_{n-2,n-2}\times \Delta(P_{n-1}+P_n)$
provide labeled cell complexes which support
$\mathbb{Z}^{n-1}$-graded minimal free resolutions of
$I_{n-1,n-2}.P_n$ and $I_{n-2,n-2}.(P_{n-1}+P_n)$, respectively. The
conditions of the gluing lema are satisfied for these two ideals and
the cell complexes. In fact, since the generators of $P_i$'s form a
partition of the set of variables, we have
$$[I_{n-1,n-2}.P_n]\cap [I_{n-2,n-2}.(P_{n-1}+P_n)]$$
$$= [I_{n-1,n-2}\cap P_n]\cap [I_{n-2,n-2}\cap (P_{n-1}+P_n)]$$ $$=
I_{n-2,n-2}.P_n.$$ On the other hand, replacing $\Gamma_{n-1,n-2}$
and $\Gamma_{n-2,n-2}$ by induction hypothesis and (a), we get
$$[\Gamma_{n-1,n-2}\times \Delta(P_n)]\cap [\Gamma_{n-2,n-2}\times
\Delta(P_{n-1}+P_n)] = \Gamma_{n-2,n-2}\times \Delta(P_n),$$ while
the last equality follows by the inclusion $\Gamma_{n-2,n-2}\subset
\Gamma_{n-1,n-2}$ which is a consequence of the expressions of them
using induction hypothesis. Therefore,
$$[\Gamma_{n-1,n-2}\times \Delta(P_n)]\cup [\Gamma_{n-2,n-2}\times
\Delta(P_{n-1}+P_n)] = \Gamma_{n,n-1}$$ is a
$\mathbb{Z}^{n-1}$-graded minimal free resolution of $I_{n,n-1}$.\\
Finally, the inclusion $\Gamma_{n,n-1}\subset \Gamma_{n,n}$ follows
by the expression of the cell complexes in (b).

\end{proof}

To prove the main result, we will further need the following lemma.

\begin{lem}\label{TLemma}
For integers $1\le q\le p\le n$ and $1\le r\le n$, let
$$I_{p,q} = \sum_{1\le j_1< \cdots <j_q\le p}P_{j_1}.\ldots .P_{j_q},$$
and
$$Q_r = \sum_{j = r}^n P_j.$$
Then;\\
(a) For all $t$, $2\le t\le n$, the following holds
$$I_{n,t} = \sum_{i=t-1}^{n-1} I_{i,t-1}.Q_{i+1}.$$
(b) For all $s$, $t\le s\le n$, let
$$I_{n,t}(s) = \sum_{i=t-1}^{s-1} I_{i,t-1}.Q_{i+1}.$$ Then for all
$s$, $t\le s \le n-1$, we have
$$I_{n,t}(s)\cap [I_{s,t-1}.Q_{s+1}] = I_{s-1,t-1}.Q_{s+1}.$$
(By (a), $I_{n,t}(s)$ is the transversal monomial ideal of degree
$t$ on $P_1, \cdots, P_{s-1}$, and $Q_s = P_s + \cdots + P_n$.)
\end{lem}

\begin{proof} For $1\le r \le s \le n$, let
$$Q_r(s) = \sum_{j=r}^s P_j,$$
so that in particular, $Q_r(n)= Q_r$. \\
To prove (a) we use induction on $n$. For $n = t$ the equality is
clear. Assuming the quality for integers less than $n$, the
inclusion $\supseteq$ is obvious for $n$ because each term in the
sum is contained in $I_{n,t}$. For the other inclusion, observe that
$$I_{n,t} = I_{n-1,t} + I_{n-1,t-1}.Q_n.$$
By induction hypothesis, we have
$$I_{n-1,t} = \sum_{i=t-1}^{n-2}I_{i,t-1}.Q_{i+1}(n-1).$$
Replacing this in the previous equality the inclusion $\subseteq$
also follows by term by term inclusion.\\
(b). By (a) we get
$$I_{n,t}(s) = \sum_{i=t-1}^{s-1}
I_{i,t-1}.Q_{i+1}(s) + (\sum_{i=t-1}^{s-1} I_{i,t-1}).Q_{s+1} =
I_{s,t} + I_{s-1,t-1}.Q_{s+1}.$$ Since {\it the modular law} holds
for monomial ideals, we have
$$I_{n,t}(s)\cap [I_{s,t-1}.Q_{s+1}] = [I_{s,t} + I_{s-1,t-1}.Q_{s+1}]\cap
[I_{s,t-1}.Q_{s+1}]$$
$$ = I_{s,t}\cap [I_{s,t-1}.Q_{s+1}] +
[I_{s-1,t-1}.Q_{s+1}]\cap [I_{s,t-1}.Q_{s+1}].$$ Since $I_{s-1,t-1}
\subseteq I_{s,t-1}$, the resulting expression is equal to
$$I_{s,t}\cap [I_{s,t-1}.Q_{s+1}] + [I_{s-1,t-1}.Q_{s+1}].$$
But $I_{s,t-1}.Q_{s+1} = I_{s,t-1}\cap Q_{s+1}$ and $I_{s,t}\cap
Q_{s+1} = I_{s,t}.Q_{s+1}$ because the set of variables involved in
the generators of $I_{s,t-1}$ or $I_{s,t}$ is disjoint from the the
set of variable generators of $Q_{s+1}$. Thus by the fact that
$I_{s,t} \subseteq I_{s,t-1}$ we get
$$I_{s,t}\cap [I_{s,t-1}.Q_{s+1}] = I_{s,t} \cap I_{s,t-1}\cap Q_{s+1} =
I_{s,t}\cap Q_{s+1} = I_{s,t}.Q_{s+1}.$$ Therefore, using the
inclusion $I_{s,t} \subseteq I_{s-1,t-1}$, we have
$$I_{n,t}(s)\cap [I_{s,t-1}.Q_{s+1}] = [I_{s,t}.Q_{s+1}] +
[I_{s-1,t-1}.Q_{s+1}] = I_{s-1,t-1}.Q_{s+1},$$ as required.
\end{proof}

We may now give the main result of the paper.

\begin{thm}\label{TThm}
There exists a labeled regular polytopal cell complex $\Gamma_{n,t}
\subset \mathbb{R}^{m-t}$ that supports a $\mathbb{Z}^t$-graded
minimal free resolution of $I_{n,t}$. Explicitly,
$$\Gamma_{n,t} = \bigcup_{i=t-1}^{n-1}[\Gamma_{i,t-1} \times
\Delta(Q_{i+1})],$$ where by $\Delta(Q_{i+1})$ we mean the simplex
with vertices on the variables generating $Q_{i+1} = \sum_{j =
i+1}^n P_j.$ Moreover, $\Gamma_{s,t}$ is connected, and as labeled
polytopes, $\Gamma_{s,t} \subset \Gamma_{s+1,t}$ for all $s$, $t\le
s\le n-1$.
\end{thm}

\begin{proof} Observe that the cell complex $\Gamma_{n,t} $ defined
recursively as above, can also be labeled recursively. Thus we first
show the inclusion $\Gamma_{k,t}  \subset \Gamma_{k+1,t}$ as labeled
cell complexes, for all $k$, $t\le k\le n-1$. Using the notation of
the previous lemma we have
$$\Gamma_{k,t} = \bigcup_{i=t-1}^{k-1}[\Gamma_{i,t-1} \times
\Delta(Q_{i+1}(k))] \subset \bigcup_{i=t-1}^k[\Gamma_{i,t-1} \times
\Delta(Q_{i+1}(k+1))] = \Gamma_{k+1,t}.$$ For all $s$, $t\le s\le
n$, similar to $I_{n,t}(s)$ defined in the previous lemma, let
$$\Gamma_{n,t}(s) = \bigcup_{i=t-1}^{s-1}[\Gamma_{i,t-1} \times
\Delta(Q_{i+1})],$$ so that $\Gamma_{n,t}(n) = \Gamma_{n,t}$. Thus
we have
$$\Gamma_{n,t}(s) = \Gamma_{n,t}(s-1)\cup [\Gamma_{s-1,t-1}\times\Delta(Q_s)].$$
Using this identity we prove by induction on $s$, $t\le s\le n$,
that $\Gamma_{n,t}(s)$ supports a $\mathbb{Z}^t$-graded minimal free
resolution of $I_{n,t}(s)$, implying the claim for $s=n$. For $s=t$,
$\Gamma_{n,t}(t) = \Gamma_{t-1,t-1}\times \Delta(Q_s)$ which support
a graded minimal free resolution of $I_{n,t}(t) = I_{t-1,t-1}.Q_s$
by Proposition \ref{TProp} and Lemma \ref{TLemma}. Assume that
$\Gamma_{n,t}(s)$ supports a graded minimal free resolution of
$I_{n,t}(s)$ for all $s$, $n$, $t$ with $t\le s<s_0\le n$. We claim
$$\Gamma_{n,t}(s_0) = \Gamma_{n,t}(s_0-1)\cup [\Gamma_{s_0 -1,t-1}\times \Delta(Q_{s_0})]$$
supports a graded minimal free resolution of
$$I_{n,t}(s_0) = I_{n,t}(s_0-1) + [I_{s_0-1,t-1}.Q_{s_0}].$$
We apply the gluing lemma. By induction hypothesis each summand of
$\Gamma_{n,t}(s_0)$ supports a graded minimal free resolution of the
corresponding summand in $I_{n,t}(s_0)$. By the previous lemma,
$$I_{n,t}(s_0-1) \cap [I_{s_0-1,t-1}.Q_{s_0}] =  I_{s_0-2,t-1}.Q_{s_0}.$$
We also have
$$\Gamma_{n,t}(s_0-1) \cap [\Gamma_{s_0-1,t-1}\times \Delta(Q_{s_0})]=
[\bigcup_{i=t-1}^{s_0-2}[\Gamma_{i,t-1} \times
\Delta(Q_{i+1})]\cap [\Gamma_{s_0-1,t-1} \times \Delta(Q_{s_0})] $$
$$= \bigcup_{i=t-1}^{s_0-2}[(\Gamma_{i,t-1}\cap \Gamma_{s_0-1,t-1})\times
(\Delta(Q_{i+1})\cap \Delta(Q_{i+1}) ] = \Gamma_{s_0-2,t-1}\times
\Delta(Q_{s_0}),$$ where the last equality is valid by Proposition
\ref{TProp}. But $\Gamma_{s_0-2,t-1}\times \Delta(Q_{s_0})$ supports
a graded minimal free resolution of $I_{s_0-2,t-1}.Q_{s_0}$, which
is what we need for the gluing lemma. Therefore, $\Gamma_{n,t}(s_0)$
supports a graded minimal free resolution of $I_{n,t}(s_0)$ as
required. The regularity of $\Gamma_{n,t}$ is clear because every
polytopal cell complex is regular \cite{Zie95}. The connectedness of
$\Gamma_{n,t}$ follows from the equality
$$\Gamma_{n,t}(s_0-1) \cap [\Gamma_{s_0-1,t-1}\times \Delta(Q_{s_0})]=
 \Gamma_{s_0-2,t-1}\times
\Delta(Q_{s_0}).$$
\end{proof}

\begin{rem}\label{compare}
The minimal free resolution provided in \cite[Theorem 3.1]{ZZ04} and
\cite[Theorem 2.1]{Zaa11} for a transversal monomial ideal is not
polytopal for $2 \le t \le n-3$. As an example, consider $I_{5,2}$
($m=n=5$). Following the construction there, the corresponding cell
complex is the same as the Morse complex which supports a minimal
free resolution of $(y_1, y_2, y_3, y_4)^2 \subset \mathbf{k}[y_1,
y_2, y_3, y_4]$ illustrated in \cite[Fig. 3]{Sin08} which is not
polytopal. In fact, the two ``faces'' $\{x_1x_2, x_1x_3, x_3x_5,
x_2x_3\}$ and $\{x_1x_2, x_1x_3, x_3x_4, x_2x_3\}$ have two edges in
common but their union is not a face of either of them.
\end{rem}

\begin{exam}\label{TExam} Consider $I_{4,3}$ with $b_1 = b_2 = 2$,
$b_3 = b_4 = 1$. Then\\

\noindent $\Gamma_{4,3} = [\Delta(x_{11},x_{12},x_{21},x_{22})\times
\Delta(x_{31})\times \Delta(x_{41})]\cup
[\Delta(x_{11},x_{12})\times \Delta(,x_{21},x_{22},x_{31})\times
\Delta(x_{41})]\cup [\Delta(x_{11},x_{12})\times
\Delta(,x_{21},x_{22})\times \Delta(x_{31},x_{41})].$\\

\noindent This is illustrated in Figure 1.\\

\begin{tikzpicture}
\draw[black, ultra thick] (-2.3,4.2) -- (-3,0.75); \draw[black,
ultra thick] (-3,0.75) -- (0.6,0.75); \draw[black, ultra thick]
(0.6,0.75) -- (-2.3,4.2); \draw[black, ultra thick] (0.6,0.75) --
(-0.45,3.75); \draw[gray, densely dotted] (-0.45,3.75) -- (-3,0.75);
\draw[black, ultra thick] (-0.45,3.75) -- (-2.3,4.2); \draw[red,
ultra thick] (-0.45,3.75) -- (0.75,6.75); \draw[red, ultra thick]
(0.75,6.75) -- (1.8,3.75); \draw[red, ultra thick] (1.8,3.75) --
(0.6,0.75); \draw[red, ultra thick] (1.8,3.75) -- (3.75,0.75);
\draw[red, ultra thick] (3.75,0.75) -- (0.6,0.75); \draw[red,
densely dotted] (0.75,6.75) -- (2.55,3.75); \draw[red, densely
dotted] (2.55,3.75) -- (3.75,0.75); \draw[red, densely dotted]
(-0.45,3.75) -- (2.55,3.75); \draw[blue, ultra thick] (0.75,6.75) --
(4.2,8.7); \draw[blue, ultra thick] (4.2,8.7) -- (5.25,5.7);
\draw[blue, ultra thick] (5.25,5.7) -- (1.8,3.75); \draw[blue, ultra
thick] (5.25,5.7) -- (7.2,2.7); \draw[blue, ultra thick] (7.2,2.7)
-- (3.75,0.75); \draw[blue, ultra thick] (7.2,2.7) -- (6,5.7);
\draw[blue, ultra thick] (6,5.7) -- (4.2,8.7); \draw[blue, densely
dotted] (6,5.7) -- (2.55,3.75); \filldraw[black] (0.6,0.75) circle
(2pt) node[anchor=north west] {$x_{12}x_{31}x_{41}$};
\filldraw[black] (-2.3,4.2) circle (2pt) node[anchor=south]
{$x_{22}x_{31}x_{41}$}; \filldraw[black] (-3,0.75) circle (2pt)
node[anchor=north] {$x_{21}x_{31}x_{41}$}; \filldraw[black]
(-0.45,3.75) circle (2pt) node[anchor=south west]
{$x_{11}x_{31}x_{41}$}; \filldraw[black] (0.75,6.75) circle (2pt)
node[anchor=east] {$x_{11}x_{21}x_{41}$}; \filldraw[black]
(2.55,3.75) circle (2pt) node[anchor=west] {$x_{11}x_{22}x_{41}$};
\filldraw[black] (3.75,0.75) circle (2pt) node[anchor=north west]
{$x_{12}x_{21}x_{41}$}; \filldraw[black] (4.2,8.7) circle (2pt)
node[anchor=south] {$x_{11}x_{21}x_{31}$}; \filldraw[black]
(5.25,5.7) circle (2pt) node[anchor=south east]
{$x_{12}x_{22}x_{31}$}; \filldraw[black] (6,5.7) circle (2pt)
node[anchor=west] {$x_{11}x_{22}x_{31}$}; \filldraw[black]
(1.8,3.75) circle (2pt) node[anchor=north east]
{$x_{12}x_{22}x_{41}$}; \filldraw[black] (7.2,2.7) circle (2pt)
node[anchor=east] {$x_{12}x_{21}x_{31}$};
\end{tikzpicture}
$$Figure \ 1$$
\vskip 0.5cm \noindent The minimal free resolution of $I_{4,3}$
(ignoring the fine grading) is immediate from Figure 1 as
$$ 0 \longrightarrow S^3(-5) \longrightarrow S^{22}(-4)
\longrightarrow S^{12}(-3) \longrightarrow I_{4,3} \longrightarrow
0.$$ \noindent The Herzog-Takayama resolution for $I_{4,3}$ requires
11 consecutive mapping cone constructions.\\
This example reveals that, in general, the cell complex
$\Gamma_{n,t}$ is not convex, and its underlying topological space
is not homeomorphic to a closed ball.
\end{exam}
For $m=n$, i.e., $b_1=\cdots =b_n=1$, $I_{n,t}$ is the {\it
square-free Veronese ideal of degree} $t$ on the variables. To
distinguish this case, we use the notation $x_i$, $I_{m,t}$ and
$\Gamma_{m,t}$ instead of $x_{i1}$, $I_{n,t}$ and $\Gamma_{n,t}$,
respectively. In this case, $\Gamma_{m,t}$ enjoys the following
further features.

\begin{prop}\label{VProp}
The polytopal cell complex $\Gamma_{m,t}$ which supports the
$\mathbb{Z}^t$-graded minimal free resolution of the square-free
Veronese ideal of degree $t$ on $x_1,\cdots, x_m$, is a polytopal
subdivision of the $(m-t)$-simplex
$$\Delta(x_1\cdots x_t, x_2\cdots x_{t+1}, \cdots, x_{n-t+1}\cdots x_m).$$
In particular, $\Gamma_{m,t}$ is shellable and its underlying
topological space is homeomorphic to a closed ball.
\end{prop}

\begin{proof}
The proof is by induction on $t$. For $t = 1$, it is trivial;
$\Gamma_{m,1} = \Delta(x_1, \cdots, x_m)$. We also need to check the
claim for $t = 2$. Adopting the notation of Lemma 3.2, we have
$$ \Gamma_{m,2} = \bigcup_{i=1}^{m-1}[\Gamma_{i,1} \times
\Delta(Q_{i+1})] = \bigcup_{i=1}^{m-1}[\Delta(x_1, \cdots,
x_i)\times \Delta(Q_{i+1})].$$ Thus, using the notation $| |$ for
the underlying topological spaces of the polytopes in
$\mathbb{R}^{m-t}$, we have
$$|\Gamma_{m,2}| = \bigcup_{i=1}^{m-1}|\Delta(x_1, \cdots, x_i)|\times
|\Delta(Q_{i+1})| = |\Delta(x_1x_2, x_2x_3, \cdots, x_{m-1}x_m)|,$$
where the last equality is a natural partition of the
$(m-2)$-simplex on $$\{x_1x_2, x_2x_3, \cdots, x_{m-1}x_m\}$$ into
the union of polytopes while each polytope is the product of an
$(i-1)$-simplex on $\{x_1, \cdots, x_i\}$ with an $(m-i-1)$-simplex
on $\{x_{i+1}, \cdots, x_m\}$, $i = 1, \cdots, m-1$. Thus,
$\Gamma_{m,2}$ is a regular subdivision of the $(m-2)$-simplex
$\Delta(x_1x_2, x_2x_3, \cdots, x_{m-1}x_m)$. For the induction
step, by the previous theorem $$\Gamma_{m,t} =
\bigcup_{i=t-1}^{m-1}[\Gamma_{i,t-1} \times \Delta(Q_{i+1})].$$ By
induction hypothesis $$|\Gamma_{m,t}| =
\bigcup_{i=t-1}^{m-1}[|\Delta(y_{t-1},\cdots, y_i)| \times
|\Delta(Q_{i+1})|],$$ where, $y_i=x_{i-t+2}\cdots x_i$,
$i=t-1,\cdots, m-1$. Thus by the case $t=2$, $$|\Gamma_{m,t}| =
|\Delta(y_{t-1}x_t, y_tx_{t+1},\cdots, y_{m-1}x_m)| =
|\Delta(x_1\cdots x_t, x_2\cdots x_{t+1}, \cdots, x_{m-t+1}\cdots
x_m)|$$ as required.\\ Shellability of $\Gamma_{n,t}$ follows from a
general result that any regular subdivision of a simplex is
shellable \cite[page 243]{Zie95}. The underlying topological space
of $\Gamma_{n,t}$ is a closed ball since it is a regular subdivision
of a simplex.
\end{proof}

\begin{rem}\label{VRem} Proposition \ref{VProp} also follows by a result of
Sinefakopoulos \cite[Theorem 12]{Sin08}. In fact, let $s = n-t+1$and
let $J = (y_1, \cdots, y_s)^t \subset \mathbf{k}[y_1, \cdots, y_s]$.
The depolarization given by Nagel and Reiner (see Section 3
\cite{NR09})
$$x_{i_1}x_{i_2}\cdots x_{i_t} \mapsto y_{i_1}y_{i_2-1}\cdots
y_{i_t-t+1}, \ \ i_1<i_2<\cdots <i_t$$ is a bijection on the minimal
monomial generators of $I_{n,t}$ and $J$. The inverse map is
$$y_1^{j_1}y_2^{j_2}\cdots
y_s^{j_s} \mapsto x_1x_2\cdots x_{j_1}x_{j_1+2}x_{j_1+3}\cdots
x_{j_1+j_2+1}\cdots x_{j_1+\cdots +j_{s-1}+s}\cdots x_{j_1+\cdots
+j_s+s-1},$$ where $j_1+\cdots +j_s = t$. The change of labeling
$x_{i_1}x_{i_2}\cdots x_{i_t}$ to $y_{i_1}y_{i_2-1}\cdots
y_{i_t-t+1}$ converts the polytopal cell complex $\Gamma_{n,t}$ into
the polytopal cell complex $P_t(y_1, \cdots, y_s)$ constructed by
Sinefakopoulos which supports the resolution of $J$ (see the proof
of \cite[Theorem 12]{Sin08} for notations). Observe that under this
change of labeling, the subdivision structures of $\Gamma_{n,t}$ on
the simplex $\Delta(x_1\cdots x_t, \cdots, x_{n-t+1}\cdots x_n)$
converts to that of $P_t(y_1, \cdots, y_s)$ on $\Delta(y_1^t,
\cdots, y_s^t).$
\end{rem}

\begin{rem} For $m>n$ the ideal $I_{n,t}$ is not Borel-fixed (i.e.,
strongly stable for characteristic zero) under certain
depolarization similar to Remark \ref{VRem}. For example, consider
$I_{2,2}$ with $b_1 = b_2 = 2$. Then it follows that $I_{2,2}$ is
not strongly stable under the expected depolarization.
\end{rem}
\noindent {\bf Question 3.9.} Sinefakopoulos \cite{Sin08} has used
the gluing procedure to build up a polytopal cell complex which
gives a minimal free resolution for the Borel-fixed monomial ideals
which are Borel-fixed generated in one degree. In this paper similar
procedure has been employed for the transversal monomial ideals. One
expects other monomial ideals for which the gluing construction
provides a minimal free resolution. In particular, we may ask
whether one could construct a similar cell complex to the
Herzog-Takayama resolution, and if possible, when is the resulting
cell complex homeomorphic to a closed ball?

\section*{Acknowledgments}
The author is grateful to Gunnar Fl{\o}ystad for his comments on the
first attempt towards this paper, for bringing attention to
\cite{Sin08} and \cite{NR09}, and for suggesting to work on the
cellularity of transversal monomial ideals. He thanks Afshin
Goodarzi for his useful comments on the final version of the paper.
This work has been partially supported by research grant no.
4/1/6103011 of University of Tehran.

\end{document}